\documentclass[12pt, a4paper]{amsart}
\usepackage{amsmath, amssymb,amsthm}

\usepackage[all]{xy}
\usepackage{tabularx}

\newtheorem{thm}{Theorem}[section]
\newtheorem{prop}[thm]{Proposition}
\newtheorem{lemma}[thm]{Lemma}
\newtheorem{cor}[thm]{Corollary}
\theoremstyle{definition}
\newtheorem{defn}[thm]{Definition}
\newtheorem{rem}[thm]{Remark}

\newcommand{\C}{\mathbb{C}}

\newcommand{\Q}{\mathbb{Q}}

\newcommand{\Z}{\mathbb{Z}}

\newcommand{\cC}{{\mathcal{C}}}
\newcommand{\cD}{{\mathcal{D}}}
\newcommand{\cM}{\mathcal{M}}
\newcommand{\cO}{\mathcal{O}}

\newcommand{\cR}{\mathcal{R}}
\newcommand{\cT}{\mathcal{T}}
\newcommand{\cU}{\mathcal{U}}
\newcommand{\cW}{\mathcal{W}}
\newcommand{\gm}{\mathfrak{m}}
\newcommand{\full}{\mathrm{full}}

\newcommand{\mat}[1]{\left( \begin{matrix} #1 \end{matrix} \right)}
\newcommand{\smat}[1]{ \left( \begin{smallmatrix} #1 \end{smallmatrix} \right)}

\DeclareMathOperator{\ad}{ad} 
 
\DeclareMathOperator{\End}{End}
\DeclareMathOperator{\Frob}{Frob}
\DeclareMathOperator{\Ind}{Ind}
\DeclareMathOperator{\Gal}{Gal}
\DeclareMathOperator{\GL}{GL}
\DeclareMathOperator{\PGL}{PGL}

\DeclareMathOperator{\rH}{H}
\DeclareMathOperator{\Hom}{Hom}
\DeclareMathOperator{\Spec}{Spec}
\DeclareMathOperator{\Tr}{Tr}

\title{On the eigencurve at classical weight one points}

\author{Jo{\"e}l Bella{\"\i}che \and Mladen Dimitrov }

\address{Brandeis University,  Department of Mathematics MS 050, 415 South Street,
Waltham, MA 02453,  USA}
\email{jbellaic@brandeis.edu}

\thanks{The first author is supported in part by the NSF grants DMS 1101615 and 1405993. The second author is supported in part by Agence Nationale de la Recherche grants ANR-10-BLAN-0114 and ANR-11-LABX-0007-01.} 
\address{Universit{\'e}  Lille 1, UMR CNRS 8524, UFR Math{\'e}matiques, 
59655 Villeneuve d'Ascq Cedex, France}
\email{mladen.dimitrov@gmail.com}

\begin{document}

\begin{abstract}
We show that the $p$-adic eigencurve is smooth at classical weight one points which are regular at $p$ and give a precise criterion for 
etaleness over the weight space at those points. Our approach uses deformations of Galois representations.  
\end{abstract}

\maketitle

\section{Introduction}\label{intro}

A  well known result of Hida \cite[Corollary 1.4]{hida85} asserts that the  eigencurve is \'etale over the weight space 
at all classical  ordinary points of weight two or more,  hence is smooth at those points. 
As a consequence of Coleman's control theorem \cite[Theorem 6.1]{coleman}, this result generalizes
to all non-critical regular classical points of weight two or more.
The aim of this paper is to further investigate the geometry of the eigencurve at classical  points of weight one. Note that these points are both ordinary and numerically critical (cf.~\cite[\S2.4.3]{bellaiche-chenevier-book})

In order to state our results, we need to fix some notations. 
Denote by $G_L$ the absolute Galois group of a perfect field $L$.   
We let $\bar\Q\subset \C$ be the field of algebraic numbers and $\tau\in G_{\Q}$
 be the complex conjugation.

Let $f(z)=\sum_{n\geq 1} a_n e^{2i \pi n z}$ be a newform of  weight one, level $M$ and nebentypus $\varepsilon$.  It is a theorem of Deligne and Serre  \cite{deligne-serre} that there exists a  continuous irreducible representation with finite image:
\begin{equation*}
\rho: G_{\Q}\to \GL_2(\C),
\end{equation*}
which is unramified outside $M$ and
such that for all $\ell \nmid M$ we have $\Tr \rho(\Frob_\ell)=a_\ell$ and $\det \rho(\Frob_\ell)=\varepsilon(\ell)$, where  $\Frob_\ell$ denotes an arithmetic Frobenius  at $\ell$.
Moreover $\rho$  is {\it odd} in the sense that $\det\rho(\tau)=-1$.

Throughout this article, we fix a prime number $p$, as well as an algebraic closure $\bar \Q_p$ of $\Q_p$ and an embedding $\iota_p:\bar\Q\hookrightarrow\bar\Q_p$.  
Let $\alpha_p$ and $\beta_p$ be the roots (possibly equal) of the  Hecke polynomial  $X^2-a_p X+\varepsilon(p)$, 
that  we see  as elements of $\bar \Q_p$ using $\iota_p$.
 We say that  $f$ is {\it regular} at $p$ if $\alpha_p\neq\beta_p$.

In order to    deform $f$   $p$-adically, one must first choose a {\it  $p$-stabilization }
of $f$ with finite slope, that is an eigenform  of level $\Gamma_1(M) \cap \Gamma_0(p)$  sharing the same eigenvalues as $f$ away from $p$ and having a non-zero $U_p$-eigenvalue.  This is done differently according to whether $p$ divides the level or not. If $p \nmid M$, we define
$f_\alpha(z)=f(z)-\beta_p f(pz)$ and $f_\beta(z)=f(z)-\alpha_p f(pz)$. Those forms are the $p$-stabilizations of $f$ with 
  $U_p$-eigenvalues $\alpha_p$ and $\beta_p$, respectively. Note that $f_\alpha$ and $f_\beta$ are distinct if, and only if, $f$ is regular at $p$.
If $p$ divides $M$ and $a_p\neq 0$ (that is if $f$ is regular at $p$), then  $U_p\cdot f=a_p f$ and we simply set $\alpha_p=a_p$ and $f_\alpha=f$. 
It follows from \cite[Theorem 4.2]{deligne-serre} that the $U_p$-eigenvalue of any $p$-stabilization $f_\alpha$ is a root of unity and {\it a fortiori} a
$p$-adic unit. Hence, if $f$ is regular at $p$, then $f_\alpha$  is  ordinary at $p$.

Denote by $N$  the prime to $p$-part of $M$. Let  $\cC$ be the eigencurve of tame level $N$, constructed 
with the Hecke operators $U_p$ and $T_\ell$  for  $\ell \nmid Np$. There exists a flat and locally finite $\kappa : \cC \to \cW$, where the weight space $\cW$ is the 
rigid analytic space over $\Q_p$ representing continuous group homomorphisms from $\Z_p^\times\times(\Z/N\Z)^\times$ to $\mathbb{G}_m$, and $f_\alpha$ defines a point on $\cC$, whose image by $\kappa$ has finite order (cf.~\S\ref{hecke} for more details).

  \begin{thm} \label{main-thm}
  Let $f$ be a classical weight one newform of level $\Gamma_1(M)$ which is regular at $p$.
  Then the eigencurve $\cC$ is smooth at $f_\alpha$. 
  Moreover $\kappa$ is \'etale at $f_\alpha$
 if, and only if, there does not exist a real quadratic field  $K$ in which $p$ splits
  and such that  $\rho_{|G_K}$ is reducible. 
     \end{thm}
     
     The latter case was first investigated by Cho and Vatsal  in \cite{cho-vatsal} where they show, under some additional assumptions,  that the  ramification index is exactly $2$. We have  
recently learned that Greenberg and Vatsal  have announced a  result similar to ours under the assumption that the adjoint representation is regular at $p$.

Let us  observe already that there is alternative definition of the eigencurve which also uses  the operators $U_\ell$ for $\ell \mid N$. 
However, had we decided to work with this eigencurve,  by Theorem~\ref{local-isom}, 
the results would have been exactly the same.

\begin{cor} \label{main-cor} Let $f$ be a classical weight one  newform of level $\Gamma_1(M)$, which is  regular at $p$. 
Then there is a unique irreducible component of $\cC$ containing $f_\alpha$. If $f$ has CM by a quadratic imaginary field $K$ in which $p$ is split, then all 
classical points of that component also have CM by $K$.
\end{cor}
We shall deduce this result at the end of section~\ref{main}.

\medskip
Our interest in studying the geometry of the eigencurve at classical weight one points arose from questions about specializations of primitive Hida families in weight one. In that language the above corollary says that there exists a unique  Hida family specializing to a given $p$-stabilized, regular, classical weight one newform. If one relaxes the condition of regularity, then there may exist distinct primitive Hida families specializing to the same $p$-stabilized classical weight one newform 
(cf.~\cite[\S 7.4]{dimitrov-ghate}).  Studying the geometry of the eigencurve
at the corresponding point not only could answer the question of how many families specialize to that eigenform, but could also potentially give a deeper insight on how exactly those families meet. 

\medskip

Another motivation is the following application to the theory of $p$-adic $L$-functions. Using some ideas of Mazur, Kitagawa associated to any Hida family satisfying certain algebraic conditions (cf.~\cite[p.105]{kitagawa})
a  $p$-adic $L$-function whose specialization in any weight at least two equals, up to a $p$-adic unit, the $p$-adic $L$-function $L_p(g,s)$  attached to a $p$-stabilized ordinary eigenform $g$ of that weight by Mazur, Manin, Vishik, Amice--Velu {\it et al.} (cf.~\cite{MTT}). 
    The smoothness established in Theorem~\ref{main-thm} allows us, by a method used in \cite{Bcrit} and detailed in \cite[Corollary V.4.2]{Bbook}, to give the following  weaker but unconditional  version of  the Mazur--Kitagawa construction.  
    
\begin{cor} Denote by $\cW^+$ (resp. $\cW^-$) the subspace of $\cW$ consisting of homomorphisms sending $-1$ to $1$ (resp. to $-1$).
 Let $f$ be a classical weight one  newform of level $\Gamma_1(M)$ which is  regular at $p$. Then there exist an admissible neighborhood $\cU$ of $f_\alpha$ in $\cC$,  real constants $0<c<C$,  and two  analytic functions $L_p^\pm$ on $\cU \times \cW^\pm$, such that for every classical $g \in \cU$ of weight at least two,
 and every $s \in \cW^\pm$ 
 $$L_p^\pm(g,s) = e^\pm(g) L_p(g,s),$$
 where $e^\pm(g)$ is a $p$-adic period  such that $c < |e^{\pm}(g)|_p < C$.
 Moreover for every $g\in \cU$ the functions $s \mapsto L_p^{\pm}(g,s)$ are bounded.
\end{cor}
As explained in {\it loc. cit.}, it is easy to see that while the two functions $L_p^\pm$ are not uniquely determined by the conditions of the corollary, each of 
the functions $L_p^\pm(f_\alpha,s) $ on $\cW^\pm$ is uniquely determined up to a non-zero multiplicative constant.  The resulting function on $\cW$ is thus well-defined up to two multiplicative non-zero constants, one on $\cW^+$ and one on $\cW^-$, and seems to be a natural definition for the
analytic $p$-adic $L$-function $L_p(f_\alpha,s)$ of $f_\alpha$. 
The zeros of $L_p(f_\alpha,s)$  with their multiplicities are well-defined and 
 there are only finitely many by the boundedness property. 

\medskip
Let us now explain the main ideas behind the proof of Theorem~\ref{main-thm}. In \S\ref{deforms} we introduce a deformation problem $\cD$ for $\rho$ representable by a 
$\bar\Q_p$-algebra  $\cR$
which by \S\ref{main} surjects to the completed local ring $\cT$ of the eigencurve at $f_\alpha$.
The computation of the tangent space to $\cD$ represents an important part of the proof and shows that its 
 dimension is always $1$, hence the above surjection $\cR \rightarrow \cT$ is an  isomorphism and the eigencurve is smooth at $f_\alpha$. The question of the etaleness is treated in a parallel way by studying  the algebra of the fiber of the weight map at $f_\alpha$, which corresponds to another deformation problem $\cD'$. 
The tangent spaces to  $\cD'$ and $\cD$  are  interpreted  in \S\ref{tangent-defn} as subspaces of $\rH^1(\Q,\ad\rho)$ satisfying certain conditions.   
Using that  $\rho$ has finite image, we apply  inflation-restriction sequences in \S\ref{infres} to further view 
the tangent spaces as  subspaces of $\Hom(G_H,\bar \Q_p) \otimes \ad \rho$, where $H$ is 
 the finite Galois extension of $\Q$ cut out by the adjoint representation of $\rho$.
In \S\ref{bakerbrumer} we introduce  a natural  $\bar \Q$-subspace $\Hom(G_H,\bar \Q_p)_{\bar \Q}$ of $\Hom(G_H,\bar \Q_p)$, stable by the natural action of $\Gal(H/\Q)$, and using the Baker--Brumer theorem in transcendence theory we give a complete description of this subspace as a $\Gal(H/\Q)$-representation.
The crucial  Proposition~\ref{cocycle} asserts   that our tangents spaces have  basis of elements
 lying in $\Hom(G_H,\bar \Q_p)_{\bar \Q} \otimes \ad \rho$.   Computing their dimensions becomes then a pleasant exercise in representation theory of finite groups, which is done in \S\ref{tangent}.

\section{Deformation theory}

We keep the notations from \S\ref{intro}. 
Since the representation $\rho$ has finite image, it is equivalent to a representation whose image is in $\GL_2(\bar \Q)$ and using the embedding $\iota_p$ we can see $\rho$ as a representation of 
$G_{ \Q}$ on a $\bar \Q_p$-vector space $V$ of dimension $2$. 

In this section we introduce and prepare the study of two deformation problems of $\rho$, that we denote by $\cD$ and $\cD'$. Their tangent spaces will be determined in  \S\ref{tangent}.  
 
\subsection{Two deformation problems}\label{deforms}

 Observe that $\iota_p$ yields a canonical embedding $G_{\Q_p} \hookrightarrow G_\Q$ through which we will see $G_{\Q_p} $ as a decomposition group of $G_\Q$ at $p$.  We denote by $I_p$ the inertia subgroup of $G_{\Q_p}$.

Assume that $f$ is regular at $p$. Then  $\rho_{|G_{\Q_p}}$ is an extension of an unramified character 
$\psi''$ such that $\psi''(\Frob_p)=\alpha$, by a {\it distinct} character $\psi'$. 
Since $\rho$ has finite image, we can fix a basis $(e_1,e_2)$ of $V$ for which $\rho(G_\Q) \subset \GL_2(\bar \Q)$
 and such that $\rho_{|G_{\Q_p}}$ acts diagonally, by the character $\psi'$ on $\bar \Q_p e_1$ and by $\psi''$ on $\bar \Q_p e_2$ (the basis $(e_1,e_2)$ is well-defined up to a scaling $(e_1,e_2)\mapsto (\lambda e_1,\mu e_2), \lambda,\mu \in \bar \Q^\ast$ which will be immaterial to us).

 We consider the following  deformation problem of the couple $(\rho,\psi'')$: 
\begin{defn}
For $A$ any local Artinian ring with maximal ideal $\gm_A$ and residue field $A/\gm_A=\bar \Q_p$, let  $\cD(A)$ be the set of strict equivalence classes of representations $\rho_A : G_{\Q} \to \GL_2(A)$ such that  $\rho_A \mod{\gm_A} = \rho$ and which are  ordinary at $p$
in the sense that:
\begin{equation} \label{condord} 
\rho_{A|G_{\Q_p}} = \left( \begin{matrix} \psi'_{A} &  * \\ 0 & \psi''_{A} \end{matrix} \right),  
\end{equation} 
where $\psi''_{A} : G_{\Q_p} \to A^\times$ is an unramified character lifting $\psi''$. 

 Let $\cD'$ be the subfunctor of $\cD$ of deformations with constant determinant. We call $t_{\cD'}\subset t_{\cD}$ the tangent spaces to those functors.
\end{defn}

In the next two sections, we will compute these tangent spaces and prove:
\begin{thm} \label{comptangent} One always has $\dim t_{\cD}=1$. One has $\dim t_{\cD'}=1$ if, and only if,  there exists a real quadratic field $K$ in which $p$ splits such that $\rho_{|G_K}$ is reducible.
\end{thm}
Since $\cD'$ is a subfunctor of $\cD$, we have $t_{\cD'}=0$ unless there exists a real quadratic field $K$ in which $p$ splits such that $\rho_{|G_K}$ is reducible.

The motivation for computing these spaces
is that, as we shall eventually see in \S\ref{main}, 
 $t_{\cD}$ can be identified with the tangent space to the eigencurve at the point corresponding to  
 $f_\alpha$, while $t_{\cD'}$ can be identified with the tangent space at that point to the fiber of the eigencurve over the weight space.
Of course,  $\dim t_{\cD}=1$ is equivalent to the smoothness of the eigencurve at $f_\alpha$, while  $ t_{\cD'}=0$ is equivalent to its etaleness at $f_\alpha$ over the weight space.

\subsection{Tangent spaces}\label{tangent-defn}

We shall now recall the standard cohomological interpretation of the tangent spaces of these deformation functors. To this purpose
let us denote by $\ad \rho$ the adjoint representation of $\rho$, that is the representation of $G_\Q$ on the space $\End_{\bar \Q_p}(V)$ on which $g \in G_\Q$ acts by  $X \to \rho(g)X \rho(g)^{-1}$. We denote by $\ad^0\rho$ the subrepresentation of $\ad \rho$ on the subspace of trace zero endomorphisms. 

The choice of the basis $(e_1,e_2)$ of $V$  identifies $\End_{\bar \Q_p}(V)$ with $M_2(\bar \Q_p)$, and therefore defines four continuous maps $A,B,C,D: \End_{\bar \Q_p}(V) \to \bar \Q_p$ (the upper-left, upper-right, lower-left, lower-right coefficients). Of course those maps are not morphisms of $G_\Q$-representations but by definition of the basis $(e_1,e_2)$ they are morphisms of $G_{\Q_p}$-representations as follows:
\begin{equation*}
\begin{split}
A : (\ad \rho)_{|G_{\Q_p}} \to \bar\Q_p(\psi'/\psi')=\bar \Q_p \text{ , } &
 B  : (\ad \rho)_{|G_{\Q_p}} \to \bar \Q_p(\psi'/\psi''),\\
C  : (\ad \rho)_{|G_{\Q_p}} \to \bar\Q_p(\psi''/\psi') \text{ , } & D  : (\ad \rho)_{|G_{\Q_p}} \to \bar \Q_p(\psi''/\psi'')=\bar \Q_p,
\end{split}
\end{equation*}
 where $\bar \Q_p(\psi''/\psi')$ denotes the representation of dimension $1$ of $G_{\Q_p}$ on $\bar \Q_p$ through the character $\psi''/\psi'$, etc. 
By composing the restriction morphism $\rH^1(\Q,\ad \rho) \to \rH^1(\Q_p,\ad \rho)$,  
coming from the canonical embedding $G_{\Q_p}\hookrightarrow G_\Q$,  with the map $\rH^1(\Q_p, \ad \rho) \to \rH^1(\Q_p,\bar\Q_p(\psi''/\psi'))$ induced by $C$,  one obtains a homomorphism  
$$C_\ast: \rH^1(\Q,\ad \rho) \to \rH^1(\Q_p, \bar \Q_p(\psi''/\psi')).$$ 
Similarly, by composing the restriction morphism  $\rH^1(\Q,\ad \rho) \to \rH^1(I_p,\ad \rho)$ 
with the map $\rH^1(I_p, \ad \rho) \to \rH^1(I_p,\bar \Q_p)$ induced by  $D$,  one obtains a homomorphism  
$$D_\ast : \rH^1(\Q,\ad \rho) \to \rH^1(I_p,\bar \Q_p).$$ 

 \begin{lemma} \label{lemmatd} We have:
 \begin{eqnarray} t_{\cD} &=& \ker \left( \rH^1(\Q,\ad \rho) \stackrel{(C_\ast,D_\ast)}{\longrightarrow} \rH^1(\Q_p,\bar\Q_p(\psi''/\psi')) \oplus \rH^1(I_p, \bar \Q_p ) \right) \text{ and } \\
 t_{\cD'} &=& \ker \left( \rH^1(\Q,\ad^0 \rho)\stackrel{(C_\ast,D_\ast)}{\longrightarrow} \rH^1(\Q_p,\bar\Q_p(\psi''/\psi')) \oplus \rH^1(I_p, \bar \Q_p ) \right).
 \end{eqnarray}
\end{lemma}   
\begin{proof} 
 By definition $t_{\cD}=\cD(A)$ and $t_{\cD'}=\cD'(A)$, where 
$A=\bar \Q_p[\epsilon]/(\epsilon^2)$ is  the algebra of dual numbers over $\Q_p$.
Any deformation $\rho_A $ of $\rho$  is of the form $ (\mathrm{Id} + \epsilon X)\rho$ where $X: G_\Q \rightarrow 
\End_{\bar \Q_p}(V)$ is a crossed homomorphism ($X(gg')= X(g)+ \rho(g)X(g') \rho(g)^{-1}$ for all $g,g'\in G_\Q$), yielding a
natural bijection between     $\cD(A)$ and $\rH^1(\Q,\ad \rho)$. 
Using (\ref{condord}), a direct computation shows that 
$\rho_A \in \cD(A)$ if, and only if, the corresponding cocycle belongs to $\ker(C_\ast,D_\ast)$. 
\end{proof}

\subsection{Use of inflation-restriction}\label{infres}

Let $H$ be a finite  Galois extension of  $\Q$ and let  $G=\Gal(H/\Q)$.
Let $W$ be a  left  $G_\Q$-representation on a finite dimensional $\bar\Q_p$-vector space. 
The natural left action of $G$ on $H$ yields a right action $x\mapsto g^{-1}(x)$ for which 
$\rH^1(H,W)$ becomes a left $\bar\Q_p[G]$-module. 

The embedding  $\iota_p$ singles out  a canonical place $w_0$ among the places $w$ of $H$ above $p$, and  canonical embeddings $G_{\Q_p}\supset G_{H_{w_0}} \subset  G_H$.  
The following general lemma will be used to transform the  expressions of $t_{\cD}$ and $t_{\cD'}$ in 
Lemma~\ref{lemmatd} under inflation-restriction sequences. 
\begin{lemma} 
\label{abstractinfres} Let $W_1$ (resp. $W_2$) be a  quotient of $W$ as a $G_{\Q_p}$-representation (resp. as an $I_p$-representation). 
Then the restriction morphism yields an isomorphism: 
\begin{equation*}
\begin{split}
 &\ker \left( \rH^1(\Q,W) \to  \rH^1(\Q_p,W_1) \oplus  \rH^1(I_{p},W_2) \right) \overset{\sim}{\longrightarrow} \\
  &\ker \left(\rH^1(H,W)^{G} \to \rH^1(H_{w_0},W_1) \oplus \rH^1(I_{w_0},W_2) \right).  
\end{split}
\end{equation*}
 \end{lemma} 
\begin{proof}
One has the following commutative diagram, where the vertical maps are the restriction morphisms:
$$\xymatrix{ \rH^1(\Q,W) \ar[r]\ar[d]  &   \rH^1(\Q_p,W_1) \oplus  \rH^1(I_{p},W_2)  \ar[d] \\
\rH^1(H,W)^{G} \ar[r] & \left(\prod_{w \mid p}  \left( \rH^1(H_{w},W_1) \oplus \rH^1(I_{w},W_2) \right)\right)^{G}
}$$
The vertical maps are isomorphisms by the inflation-restriction exact sequence, since their kernels and cokernels are (products) of cohomology spaces of finite groups 
with values in characteristic zero representations. Therefore, the kernel of the two horizontal maps are isomorphic. Since $G$ acts transitively on the set of places $w$ of $H$ above $p$,  a $G$-invariant  element of  $\prod_{w \mid p}  \left( \rH^1(H_{w},W_1) \oplus \rH^1(I_{w},W_2) \right)$ whose component at one place $w_0$ is zero, is zero at all places $w$. The lemma follows.
\end{proof}

\medskip

We are now ready to apply inflation-restriction to analyze our tangent spaces. Take $H \subset \bar \Q$ to be the number field fixed by $\ker (\ad \rho)$. Then $G=\Gal(H/\Q)$ identifies naturally with the projective image of the representation $\rho$.  Using the basis $(e_1,e_2)$, an element $\rho(g)$ for $g \in G$ is
thus an element of $\PGL_2(\bar \Q)$, and for $X$ a matrix in $M_2(\bar \Q_p)$ we shall denote with a slight abuse of language by $\rho(g)X \rho(g)^{-1} $ the image of $X$ by the adjoint action of $\rho(g)$. 

Note that $\rH^1(H,\ad \rho) =\rH^1(H,\bar \Q_p) \otimes_{\bar \Q_p} \ad \rho $ since by definition $\ad \rho(G_H)=1$.
An element of $\rH^1(H,\bar\Q_p) \otimes \ad \rho$ can be written as a matrix $\smat{ a & b \\ c & d}$ where $a,b,c,d \in \rH^1(H,\bar\Q_p)$, and the natural left $G$-action on $\rH^1(H,\bar\Q_p) \otimes \ad \rho$ is given by: 
$$g\cdot\begin{pmatrix} a & b \\ c & d \end{pmatrix}  = 
\rho(g) \begin{pmatrix} g \cdot a & g \cdot b \\ g \cdot c & g \cdot d \end{pmatrix}\rho(g)^{-1}.$$ Thus  an element of $\rH^1(\Q,\ad \rho) \simeq \left(\rH^1(H,\bar\Q_p) \otimes \ad \rho \right)^G$
is just a matrix as above which is $G$-invariant.
Combining Lemma~\ref{lemmatd}, Lemma~\ref{abstractinfres} and the above discussion, one gets:

\begin{lemma} \label{descptcd} The morphism sending $\smat{ a & b \\ c & d}\in  \rH^1(H,\bar\Q_p) \otimes \ad \rho$ to the 
restriction of $c$ to $G_{H_{w_0}}$ and  the restriction of $d$ to $I_{w_0}$, yields the following isomorphisms: 
 \begin{eqnarray*}t_{\cD} &=& \ker \left( \left(\rH^1(H,\bar\Q_p) \otimes_{\bar \Q_p} \ad \rho\right)^G \longrightarrow  
 \rH^1(H_{w_0},\bar\Q_p)\oplus \rH^1(I_{w_0}, \bar \Q_p) \right)\\
 t_{\cD'} &=& \ker \left( \left(\rH^1(H,\bar\Q_p) \otimes_{\bar \Q_p} \ad^0 \rho\right)^G \longrightarrow \rH^1(H_{w_0},\bar\Q_p)\oplus \rH^1(I_{w_0}, \bar \Q_p) \right).
 \end{eqnarray*}
\end{lemma}   

\begin{rem} 
Though, in view of  the proof of Lemma~\ref{abstractinfres}, one could use instead of $w_0$  any place $w$ of $H$ above $p$, only for $w=w_0$ do the morphisms of Lemma~\ref{descptcd} take the  simple form given there. In fact, 
 the decomposition subgroup $G_{\Q_p}\subset G_{\Q}$ determined by $\iota_p$, and used to define our basis $(e_1,e_2)$, contains a decomposition subgroup of $G_{H}$  at $w$ only when $w= w_0$. 
\end{rem}

\section{Applications of  the Baker--Brumer theorem} \label{bakerbrumer}

Let $H$ be a finite  Galois extension of  $\Q$ with ring of integers $\cO_H$. 

We shall denote by $\widehat{G}$ the set of equivalence classes of left irreducible representations of $G=\Gal(H/\Q)$ over $\bar\Q$, or over $\bar\Q_p$ (the two sets are identified by the embedding $\iota_p$). If $\pi \in \widehat{G}$, we denote by $\pi^+$ (resp. $\pi^-$) the subspace of  $\pi$ where $\tau$ acts by $+1$ (resp. $-1$). 
The trivial representation of $G$ is denoted by $1$.

\subsection{Reminder on local units}

One has $\cO_{H}\otimes_\Z \Z_p \simeq \prod_{{w} \mid p} \cO_{H,{w}}$, where  ${w}$ runs over all places of $H$ above $p$ and $\cO_{H,{w}}$ is the ring of integers of the completion $H_{w}$.

By local class field theory the image of the restriction homomorphism $$\Hom(G_{H_{w}}, \bar \Q_p) \to \Hom(I_{w}, \bar \Q_p)$$ is isomorphic to $\Hom (\cO_{H,{w}}^\times, \bar \Q_p)$ which we will now describe. 

Let  $\log_p: \bar \Q_p^\times \to  \bar \Q_p$ be the standard determination of the
$p$-adic logarithm sending  $p$ to $0$. 
A continuous homomorphism $\Hom (\cO_{H,{w}}^\times, \bar \Q_p)$ is of the form 
\begin{equation*}
u\mapsto 
\sum_{g_w\in J_w} h_{g_w} g_w(\log_p(u))=
\sum_{g_w\in J_w} h_{g_w} \log_p(g_w(u)),
\end{equation*}
for some $h_{g_w}\in  \bar \Q_p$, where $J_w$ is the set of all embeddings
of $H_{w}$ in $\bar \Q_p$.

The embedding $\iota_p:\bar\Q\hookrightarrow\bar\Q_p$  and the 
inclusions $H\subset \bar\Q \subset \C$ define a partition
$$G=\coprod_{w\mid p} J_w $$
coming from the following commutative diagram:
$$\xymatrix{ \bar\Q \ar@{^{(}->}^{\iota_p}[r] & \bar\Q_p \\
H\ar@{^{(}->}^{g}[u] \ar@{^{(}->}[r] & H_w. \ar@{^{(}->}^{g_w}[u] }$$
Hence the  $\bar \Q_p$-vector space  $\Hom((\cO_H \otimes_\Z \Z_p)^\ast,\bar \Q_p)$ has a canonical basis, namely 
$$(\log_p ( \iota_p \circ g\otimes 1))_{g \in G}.$$ 
Since  $g' \in G$ acts on the left on this  basis sending  $\log_p( \iota_p \circ g\otimes 1)$ to  $\log_p ( \iota_p \circ g'g\otimes 1)$,  there is a canonical isomorphism of left $G$-representations: 
\begin{equation}\label{group-alg}
\begin{split}
\bar\Q_p[G] & \overset{\sim}{\longrightarrow}  \Hom \left( (\cO_{H}\otimes \Z_p)^\times, \bar \Q_p\right) \\
\sum_{g\in G} h_{g} g  & \mapsto
\left(u\otimes v \mapsto \sum_{g \in G} h_{g} \log_p \big(\iota_p( g^{-1}(u))v  \big)\right).
\end{split}
\end{equation}

\begin{defn} \label{subring}
For $R$ any subring of $\bar\Q_p$, we define $\Hom((\cO_H \otimes_\Z \Z_p)^\ast,\bar \Q_p)_R$ as the $R$-linear span of the elements of the canonical basis 
$(\log_p( \iota_p \circ g\otimes 1))_{g \in G}$.
\end{defn}

It is clear that one has $\Hom((\cO_H \otimes_\Z \Z_p)^\ast,\bar \Q_p)_R \otimes_R \bar \Q_p = \Hom((\cO_H \otimes_\Z \Z_p)^\ast,\bar \Q_p)$ and that (\ref{group-alg}) restricts to an isomorphism of left  $R[G]$-modules: 
\begin{equation}\label{group-alg-R}
R[G]  \overset{\sim}{\longrightarrow}  \Hom \left( (\cO_{H}\otimes_\Z \Z_p)^\times, \bar \Q_p\right)_R 
\end{equation}

\subsection{Reminder on the cohomology space $\rH^1(H,\bar \Q_p)=\Hom(G_H,\bar \Q_p)$}

Recall  that by global class field theory one has an  exact sequence of 
left $\bar \Q_p[G]$-modules
\begin{equation} \label{cftes} 0 \to \Hom(G_H,\bar \Q_p) \to \Hom((\cO_H \otimes_\Z \Z_p)^\ast,\bar \Q_p) \to \Hom(\cO_H^\ast,\bar \Q_p),\end{equation}
where the first map is dual to the Artin reciprocity map, and the second is the restriction with respect to the inclusion $\cO_H^\ast \to (\cO_H \otimes_\Z \Z_p)^\ast, u\mapsto u\otimes 1$. The left $G$-actions come from the right $G$-action on $H$ (cf.~\S\ref{infres}). The surjectivity of the last map in (\ref{cftes}) is an open problem, equivalent to Leopoldt's conjecture. 

By (\ref{group-alg}), $\Hom((\cO_H \otimes_\Z \Z_p)^\ast,\bar \Q_p)$  is isomorphic to the regular representation   $\displaystyle \bigoplus_{\pi \in \widehat{G}} \pi^{\dim \pi}$ of $G$, while Minkowski's proof of Dirichlet's unit theorem implies that:
\begin{equation}\label{minkowski}
\cO_H^\ast \otimes_\Z \bar \Q \simeq \bigoplus_{\pi \in \widehat{G}, \pi \neq 1} \pi^{\dim \pi^{+}}.
\end{equation}
Hence, as a left $G$-representation, one has
\begin{equation} \label{mpi} \Hom(G_H,\bar \Q_p) \simeq \bigoplus_{\pi \in \widehat{G}} \pi^{m_\pi},\text{ with }m_1=1\text{ and }  \dim \pi^- \leq m_\pi \leq \dim \pi \text{ if }\pi \neq 1. \end{equation}
The Leopoldt's conjecture for $H$ at the prime $p$ is equivalent to the equality $m_\pi=\dim \pi^-$ for every non-trivial $\pi$. The theorem of Baker--Brumer \cite{brumer} on the $\bar \Q$-linear independence of $p$-adic logarithms of algebraic numbers allows to slightly improve (\ref{mpi}), namely to prove (as in \cite[proof of Theorem 1]{emsalem}) that 
\begin{equation} \label{mpi3} m_\pi < \dim \pi\text{ , if } \pi\neq 1 \text{ and } \pi^+\neq 0.\end{equation} 
Note in particular that when $G$ is abelian, $\dim \pi =1$ for all $\pi$ so (\ref{mpi}) and (\ref{mpi3}) imply $m_\pi = \dim \pi^-$ for every non-trivial $\pi$, and we retrieve Ax's well known result that Leopoldt's conjecture holds in the abelian case. 

\begin{lemma} \label{suggestion}
The dimension of the space $\rH^1(\Q,\ad^0 \rho)$ is $2$. 
\end{lemma}
\begin{proof} We have $\rH^1(\Q,\ad^0 \rho)=\left( \Hom(G_H,\bar \Q_p) \otimes_{\bar \Q_p} \ad^0 \rho \right)^G$, 
where $H$ denotes the fixed field of $\ad^0\rho$. Using the fact that each irreducible component of $\ad^0 \rho$ is non-trivial, self-dual and without multiplicity (cf.~\S\ref{tangent}), Schur's lemma implies that: 
$$\dim \left( \Hom(G_H,\bar \Q_p) \otimes_{\bar \Q_p} \ad^0 \rho \right)^G = \sum m_\pi,$$
where $\pi$ runs over all irreducible components of $\ad^0 \rho$ (cf.~(\ref{mpi})). 
Since  $\rho$  is  odd,  the eigenvalues of $\ad^0\rho(\tau)$ are $+1,-1,-1$, therefore every component $\pi$ of $\ad^0 \rho$ satisfies either $\dim \pi^{-} = \dim \pi$ or $\dim \pi^{-} = \dim \pi-1$. According to (\ref{mpi}) and (\ref{mpi3})
 in both cases we  have $m_\pi =  \dim \pi^{-}$,  hence: 
$$\dim \left( \Hom(G_H,\bar \Q_p) \otimes_{\bar \Q_p} \ad^0 \rho \right)^G = \dim (\ad^0 \rho)^{-}=2.$$ \end{proof}
\begin{rem} The above lemma, which is well-known to specialists, says that odd two-dimensional representations of $G_\Q$ over $\bar \Q_p$ are {\it unobstructed}
in the sense of Mazur when they have finite image. It is conjectured, but not known, that even without the finite image, two-dimensional   representations of $G_\Q$ over $\bar \Q_p$
of geometric origin are always unobstructed.
\end{rem}

\subsection{The algebraic subspace $\Hom(G_H,\bar \Q_p)_{\bar \Q}$ of $\Hom(G_H,\bar \Q_p)$.}

Our computation of tangent spaces of Galois deformation problems will take place in a certain natural $\bar \Q$-subspace $\Hom(G_H,\bar \Q_p)_{\bar \Q}$ of $\Hom(G_H,\bar \Q_p)$, which we will    completely determine using the Baker--Brumer theorem. 

\begin{defn} \label{local-Qbar} We define a left $\bar \Q[G]$-module  $\Hom(G_H,\bar \Q_p)_{\bar \Q}$ as the intersection of $\Hom(G_H,\bar \Q_p)$ and $\Hom((\cO_H \otimes_\Z \Z_p)^\ast,\bar \Q_p)_{\bar \Q}$ (cf.~Definition \ref{subring}) inside 
$\Hom((\cO_H \otimes_\Z \Z_p)^\ast,\bar \Q_p)$ using  (\ref{cftes}).
 In other terms,
\begin{equation} \label{defH1R} \Hom(G_H,\bar \Q_p)_{\bar \Q} = \ker \left( \Hom((\cO_H \otimes_\Z \Z_p)^\ast,\bar \Q_p)_{\bar \Q} \to \Hom(\cO_H^\ast,\bar \Q_p) \right)
\end{equation}
\end{defn}

There is no reason to expect that $ \Hom(G_H,\bar \Q_p)_{\bar \Q} \otimes_{\bar \Q} \bar \Q_p = \Hom(G_H,\bar \Q_p)$ and  the following theorem shows that this is false in general.

\begin{thm} \label{H1Qbar} There is an isomorphism of  left $\bar \Q[G]$-modules, 
$$\Hom(G_H,\bar \Q_p)_{\bar \Q} \simeq \bigoplus_{\stackrel{\pi \in \widehat{G}}{\pi=1\text{ or }\pi^+=0}} \pi^{\dim \pi}.$$
\end{thm}

\begin{proof} A key observation is that $\Hom(G_H,\bar \Q_p)_{\bar \Q}$ carries a hidden structure of a {\it right} 
 $\bar \Q[G]$-module, that we will now unveil. 

The $\bar \Q[G]$-module structure of $\cO_H^\ast\otimes_\Z \bar \Q$ yields a structural homomorphism: 
$$\bar \Q[G] \to \Hom_{\bar \Q}(\cO_H^\ast\otimes_\Z \bar \Q,\cO_H^\ast \otimes_\Z \bar \Q)$$
 which is both left and right $G$-equivariant.  By (\ref{group-alg-R}) this defines the last map of the sequence: 
\begin{equation} \label{cftes-Q} 0 \to \Hom(G_H,\bar \Q_p)_{\bar \Q}  \to \Hom((\cO_H \otimes_\Z \Z_p)^\ast,\bar \Q_p)_{\bar \Q}  \to \Hom(\cO_H^\ast, \cO_H^\ast\otimes_\Z \bar \Q),\end{equation}
which we will now show is  exact. By (\ref{defH1R}) it suffices to show that the natural map: 
$$\cO_H^\ast\otimes_\Z \bar \Q\to  \bar \Q_p \text{ , } u\otimes v \mapsto \log_p(\iota_p(u))\iota_p(v)$$
is injective, which is a consequence of the Baker--Brumer theorem \cite{brumer} since, if $u_1, \dots u_r$ form a basis of a finite index torsion-free subgroup of $\cO_H^\ast$, the $p$-adic logarithms of algebraic numbers $\log_p(\iota_p(u_i))$, $i=1,\dots,r$ are $\Q$-linearly independent, hence also $\bar \Q$-linearly independent. Hence (\ref{cftes-Q}) is exact and in particular 
$\Hom(G_H,\bar \Q_p)_{\bar \Q}$ is both   left  and right $\bar \Q[G]$-module.

To conclude we use (\ref{minkowski}) together with the following elementary lemma: if $W$ is any left 
$\bar \Q[G]$-module, then the left $\bar \Q[G]$-module $\ker \left(  \bar \Q[G] \to \Hom_{\bar \Q}(W,W) \right)$
is isomorphic to $\bigoplus_{\pi} \pi^{\dim \pi}$, where the sum is over all $\pi \in \widehat{G}$ which does  
not appear in $W$. 
\end{proof}

\section{Computation of tangent spaces : Proof of Theorem~\ref{comptangent}}\label{tangent}

\begin{prop} \label{cottcd} One has $1 \leq \dim t_{\cD} \leq 1+\dim t_{\cD'}$. \end{prop}
\begin{proof}  Recall from Lemma~\ref{lemmatd} that
$$t_{\cD} = \ker \left( \rH^1(\Q,\ad \rho) \stackrel{(C_\ast,D_\ast)}{\longrightarrow} \rH^1(\Q_p,\bar\Q_p(\psi'/\psi'')) \oplus \rH^1(I_p, \bar \Q_p ) \right)$$ 
and that $t_{\cD'}$ has the same description with $\ad \rho$ replaced by $\ad^0 \rho$.
Since $\ad \rho = \ad^0 \rho \oplus \bar \Q_p$ and $\dim \rH^1(\Q,\bar\Q_p)=1$ it follows that 
$\dim t_{\cD} \leq 1+\dim t_{\cD'}$. 

Lemma~\ref{suggestion} implies that  $\dim\rH^1(\Q,\ad \rho)=3$. Tate's local Euler characteristic formula
yields  $\dim \rH^1(\Q_p,\bar\Q_p(\psi'/\psi'')) = 1$, hence the rank of $C_\ast$ is at most 1. 
Local Class Field Theory easily implies that the rank of the restriction morphism $\rH^1(\Q_p, \bar \Q_p )\to \rH^1(I_p, \bar \Q_p )$ is $1$, hence the rank of $D_\ast$, which factors through that morphism, is at most $1$. Therefore $t_{\cD}$ has dimension at least $3-1-1=1$. \end{proof}

Let  $H \subset \bar \Q$ be again  the number field fixed by $\ker (\ad \rho)$ and $G=\Gal(H/\Q)$.
By  lemma~\ref{descptcd} one can  write an element of   $t_{\cD}\subset \left(\Hom(G_H,\bar \Q_p) \otimes \ad \rho \right)^G$ as 
$\smat{ a & b \\ c & d }$ with $a,b,c,d\in \Hom(G_H,\bar \Q_p)$. 
 Using the natural  morphisms (\ref{cftes}) and (\ref{group-alg}) of  left  $G$-representations 
$$\Hom(G_H,\bar \Q_p) \hookrightarrow  \Hom((\cO_H \otimes_\Z \Z_p)^\ast ,\bar \Q_p) \simeq \bar \Q_p[G]$$
allows us to see
$a,b,c,d$ as elements $\sum_{g \in G} a_g g, \sum_{g \in G} b_g g, \sum_{g \in G} c_g g, \sum_{g \in G} d_g g$ of $\bar \Q_p[G]$ such that for every $g \in G$, we have 
 \begin{equation} \label{invariance}   \mat{a_{g} & b_{g} \\ c_{g} & d_{g}} = \rho(g) \mat{a_{1} & b_{1} \\ c_{1} & d_{1} } \rho(g)^{-1}. \end{equation}
 
  \begin{prop} \label{cocycle}
The tangent space  $t_{\cD'}$ has dimension at most one and is generated by an element $\smat{ a & b \\ c & d }$ of 
$\left(\Hom(G_H,\bar \Q_p)_{\bar \Q}\otimes \ad^0\rho \right)^{G}$ such that $a_1=c_1=d_1=0$.
\end{prop}
\begin{proof} Let $\smat{ a & b \\ c & d }$ be any element of $t_{\cD'}\subset \left(\Hom(G_H,\bar \Q_p) \otimes \ad^0 \rho \right)^G$ with $b_1\in \bar\Q$.
 By Lemma~\ref{descptcd} one has  $c(I_{w_0})=d(I_{w_0})=0$, which in the above notations 
 implies in particular $c_1=d_1=0$.
Since $a=-d$, one also has $a_1=0$. The elements $a_g, b_g, c_g, d_g$ are then determined by (\ref{invariance}) and belong all to  $\bar \Q$ since $\rho(g) \in \PGL_2(\bar \Q)$ for all $g\in G$. This  
  proves that $\smat{a & b \\ c &d}\in \Hom(G_H,\bar \Q_p)_{\bar \Q}\otimes \ad^0\rho$ and generates
$t_{\cD'}$. 
\end{proof}

The proof of Theorem~\ref{comptangent} will  proceed by case analysis depending on the  
representation $\ad^0 \rho$ into irreducible representations. There are three possibilities:

\begin{enumerate}
\item {\it $\ad^0 \rho$ is irreducible.} This is the so-called {\it exceptional} case, and $G$ is isomorphic to $A_4$, $S_4$ or $A_5$.

\item {\it  $\ad^0 \rho$ is the sum of two irreducible representations.} In this case $G$ is a non-abelian dihedral group and  there exists a unique quadratic  number field $K$  and  a finite order character $\chi: G_K \rightarrow \bar \Q^\ast$ such that $\rho \simeq \Ind_{G_K}^{G_\Q} \chi$. 

 Let $\sigma$ be the non-trivial automorphism of $K$ and let $\chi^\sigma$ be the character of $G_K$ defined by $\chi^\sigma(g) = \chi(\sigma^{-1} g \sigma)$. Then   $H$ is the fixed field of $\ker(\chi^\sigma/\chi)$,  and if $\varepsilon_K$ is the quadratic character of $K$, one has 
 \begin{eqnarray} \label{stradrho}
 \ad^0 \rho \simeq \varepsilon_K \oplus \Ind_{G_K}^{G_\Q} (\chi/\chi^\sigma).
 \end{eqnarray}

\item {\it $\ad^0 \rho$ is the sum of three characters.} In this case $G$ is isomorphic to $\Z/2\Z \times \Z/2\Z$ and 
there exist one real quadratic  field $K$, a finite order character $\chi: G_K \rightarrow \bar \Q^\ast$,  and two imaginary quadratic fields $K'$ and $K''$ such that 
$$\ad^0 \rho \simeq \varepsilon_K \oplus \Ind_{G_K}^{G_\Q} (\chi/\chi^\sigma) \simeq \varepsilon_K \oplus \varepsilon_{K'} \oplus \varepsilon_{K''}. $$
\end{enumerate}

\begin{proof}[Proof of Theorem~\ref{comptangent}]
We will show that $t_{\cD'}=0$ unless there exists a real quadratic field $K$ in which $p$ splits such that $\rho_{|G_K}$ is reducible, in which  case $\dim t_{\cD'}=\dim t_{\cD}=1$. 
 Proposition \ref{cottcd} would then imply that $\dim t_{\cD}=1$  in all cases.

Recall that  since $\rho$ is odd, the eigenvalues of $\ad^0\rho(\tau)$ are $+1,-1,-1$, hence $(\ad^0\rho)^+\neq0$. 
If $\ad^0\rho$ is irreducible then  $\left(\Hom(G_H,\bar \Q_p)_{\bar \Q} \otimes \ad^0 \rho \right)^G$ vanishes by Theorem~\ref{H1Qbar}, hence $t_{\cD'}=0$ by Proposition~\ref{cocycle}.

Henceforth  we assume that  $\ad^0\rho$ is reducible.  

\smallskip
Assume that we are in the above case (ii) with  $K$  imaginary.
Then $\varepsilon_K(\tau)=-1$ hence $(\Ind_{G_K}^{G_\Q} (\chi/\chi^\sigma))^+\neq 0$ and
 by Theorem~\ref{H1Qbar} the only component of $\ad^0 \rho$ contributing 
to $\Hom(G_H,\bar \Q_p)_{\bar \Q}$ is $\varepsilon_K$. It follows  that
for $\smat{ a & b \\ c & d }$ a generator of $t_{\cD'}$ as in  Proposition~\ref{cocycle}, one has
$$\begin{pmatrix}a_g & b_g \\ c_g & d_g \end{pmatrix}=
\rho(g)^{-1}\begin{pmatrix}0 & b_1 \\ 0 & 0 \end{pmatrix}\rho(g)
 =\varepsilon_{K} (g) \begin{pmatrix}0 & b_1  \\ 0 & 0 \end{pmatrix},$$ 
for all $g\in G$. If $b_1\neq0$ then the above relation shows that the subspace $\bar \Q_p e_1$ is invariant by 
$\rho(g)$ for all $g\in G$, leading to a contradiction since $\rho$ is irreducible. Hence 
$b_1=0$ and $t_{\cD'}=0$.

\smallskip
Assume finally that we are in case (iii) or in case (ii) with  $K$ real.
Then  $\varepsilon_K(\tau)=1$ and by Theorem~\ref{H1Qbar} we have 
$(\Hom(G_H,\bar \Q_p)_{\bar \Q}\otimes \varepsilon_K)^G=0$, hence
\begin{equation}\label{real}
(\Hom(G_H,\bar \Q_p)_{\bar \Q}\otimes \ad^0\rho)^G=
(\Hom(G_H,\bar \Q_p)_{\bar \Q}\otimes \Ind_{G_K}^{G_\Q} (\chi/\chi^\sigma))^G.
\end{equation}

{\it  First  subcase: $p$ splits in $K$.} 
In this case, if $v$ is the place of $K$ above $p$ determined by $\iota_p$, one has $G_{\Q_p}=G_{K_v} \subset G_K$, hence  $\rho_{|G_K}$ acts diagonally in the basis $(e_1,e_2)$. 
The isomorphism of $G_{\Q}$-representations  (\ref{stradrho}) is then realized by 
$$\begin{pmatrix}a & b \\ c &-a\end{pmatrix} \mapsto \left( \begin{pmatrix}a & 0 \\ 0 &-a\end{pmatrix}, \begin{pmatrix}0 & b \\ c& 0\end{pmatrix}\right).$$
hence by Lemma \ref{descptcd} and (\ref{real}), we  have:
 \begin{eqnarray*} t_{\cD'} &=&  \ker \left( (\Hom(G_H,\bar \Q_p)_{\bar \Q}\otimes \Ind_{G_K}^{G_\Q} (\chi/\chi^\sigma))^G \rightarrow 
 \Hom(G_{H_{w_0}}, \bar \Q_p)\right) \\
  t_{\cD} &=& t_{\cD'} \oplus \ker \left( \Hom(G_H,\bar \Q_p)^G \rightarrow \Hom(I_{w_0}, \bar \Q_p) \right).
\end{eqnarray*}
One has $\ker \left( \Hom(G_H,\bar \Q_p)^G \rightarrow \Hom(I_{w_0}, \bar \Q_p) \right) = 
\ker \left(\Hom(G_\Q,\bar \Q_p) \rightarrow \Hom(I_p,\bar\Q_p)\right)$ by Lemma~\ref{abstractinfres}, and the latter is 0 by class field theory.
Hence $t_{\cD}=t_{\cD'}$ and by Propositions~\ref{cottcd} and \ref{cocycle} we have $\dim t_{\cD} = \dim t_{\cD'}=1$.

\medskip
 {\it  Second  subcase:  $p$ is inert or ramified in $K$.}
 Since $G_{\Q_p}$ is not contained in $G_K$, its image in the dihedral group $H$ contains 
 an order $2$ element which by a slight abuse of notation we   denote $\sigma$.  
 Let $(e'_1,e'_2)$ be a basis in which $\rho_{|G_K} =  \chi \oplus \chi^\sigma$.
By rescaling this basis one can assume that $\rho(\sigma)=\left(\begin{smallmatrix}0 & 1 \\ 1 & 0 \end{smallmatrix}\right)\in \PGL_2(\bar \Q)$ so that  $\rho(\sigma)$ is diagonal in the basis $(e'_1+e'_2,e'_1-e'_2)$. We may therefore assume that $(e'_1+e'_2,e'_1-e'_2)$ is our basis $(e_1,e_2)$.  
As in the previous subcase the 
isomorphism  (\ref{stradrho}) can be realized in the basis $(e'_1,e'_2)$ as: 
$$\begin{pmatrix}a' & b' \\ c' &-a'\end{pmatrix} \mapsto \left( \begin{pmatrix}a' & 0 \\ 0 &-a'\end{pmatrix}, \begin{pmatrix}0 & b' \\ c'& 0\end{pmatrix}\right),$$
hence  is  given  in the basis  $(e_1,e_2)$  by:  
$$\begin{pmatrix}a & b \\ c &-a\end{pmatrix} \mapsto \left( \begin{pmatrix}0 & \frac{b+c}{2} \\ \frac{b+c}{2} &0 \end{pmatrix}, \begin{pmatrix}a & \frac{b-c}{2} \\ \frac{c-b}{2} &-a\end{pmatrix}\right).$$
Let $\smat{ a & b \\ c & -a }$ be a generator of $t_{\cD'}$ as in Proposition~\ref{cocycle}. 
The above computation together with (\ref{real}) implies that $b+c=0$, while 
 Lemma \ref{descptcd}  implies that  $a_1=c_1=0$. Hence    $b_1=-c_1=0$ and  $t_{\cD'}=0$, which competes the proof of the theorem.  
\end{proof}

\section{The eigencurve and its ordinary locus}\label{hecke}

As in the introduction, we let $\cC$ be the $p$-adic eigencurve of tame level $N$ constructed 
using the Hecke operators $U_p$ and $T_\ell$,  $\ell \nmid Np$. It is 
reduced and equipped with a flat and locally finite morphism $\kappa : \cC \to \cW$, called the weight map
(we refer the reader to \cite{coleman-mazur},  in the case $N=1$,  and to \cite{buzzard} for the general case). 
 By construction, there exist analytic functions $U_p\in \cO(\cC)^\times $ and $T_\ell\in \cO(\cC) $, for  $\ell \nmid N$.   The locus where $|U_p|_p=1$ is open and closed in $\cC$, and is called the ordinary locus of the eigencurve; it is closely related to Hida families.    There exists (cf.~\cite[\S7]{chenevier-thesis})  a continuous pseudo-character 
 \begin{equation}\label{pseudo}
 G_{\Q} \to \cO(\cC)
 \end{equation}
 sending  $\Frob_\ell$ to $T_\ell$ for all $\ell \nmid Np$.

If $f_\alpha$ is as in the introduction, its system of eigenvalues corresponds to a point $x\in\cC(\bar \Q_p)$ such that $\kappa(x)$ has finite order. 
Since  $U_p(x)=\alpha$ is a $p$-adic unit, $x$ actually  lies on the ordinary locus of $\cC$. 
Denote by $\cT$ the completed local ring of $\cC$ at $x$, and by $\gm$ its maximal ideal.

\begin{prop} \label{rhott} There exists a continuous representation $$\rho_{\cT} : G_{\Q} \to \GL_2(\cT),$$  such that $\Tr \rho_{\cT}(\Frob_\ell) = T_\ell$ for all $\ell \nmid Np$. 
The reduction of $\rho_{\cT}$ modulo $\gm$  is isomorphic to $\rho$. 
If $f$ is regular at $p$, then $\rho_{\cT}$ is ordinary at $p$ in the sense that:
\begin{equation*}  (\rho_{\cT})_{|G_{\Q_p}} \simeq \left( \begin{matrix} \psi'_{\cT} &  * \\ 0 & \psi''_{\cT} \end{matrix} \right),  \end{equation*} 
where $\psi''_{\cT} : G_{\Q_p} \to \cT^\times$ is the unramified character sending  $\Frob_p$ to $U_p$.
\end{prop}
\begin{proof} Consider the pseudo-character  $ G_{\Q} \to \cT$ obtained by composing 
(\ref{pseudo})  with the algebra homomorphism   $\cO(\cC) \to \cT$. It is a two dimensional pseudo-character taking values in a strictly henselian (even complete) local ring, whose residual pseudo-character is given by the  trace of the irreducible representation $\rho$. 
By a theorem of Nyssen \cite{nyssen} and Rouquier \cite{rouquier} there  exists a continuous representation $\rho_{\cT}$ as in the statement. It remains to show that $\rho_{\cT}$ is ordinary at $p$.

Denote by  $\cM$ the free rank two $\cT$-module on which  $\rho_{\cT}$ acts.  
Let  $(K_i)_{1\leq i \leq r}$ be the fields obtained by localizing $\cT$ at its minimal primes and 
put $V_i=\cM\otimes_{\cT}K_i$ ($1\leq i \leq r$). The $K_i[G_{\Q}]$-module 
 $V_i$ is nothing else but the Galois representation attached to a Hida family specializing to $f_\alpha$. By \cite[Theorem 2.2.2]{wiles}
there exists a  short exact sequence of $K_i[G_{\Q_p}]$-modules: 
\begin{equation*}
0\to  V'_i\to  V_i \to  V''_i\to 0,
\end{equation*}
where $V''_i$ is a line on which $G_{\Q_p}$ acts via the unramified character sending $\Frob_p$ to $U_p$. 
Since $\cT$ is reduced, the natural homomorphism  $\cT\to\prod_i K_i$ is injective. Since $\cM$  is free over 
$\cT$, so is the natural homomorphism $\cM \to\prod_i V_i$. This leads to a  short exact sequence of $\cT[G_{\Q_p}]$-modules:
\begin{equation}\label{splits}
0\to  \cM'\to  \cM \to  \cM''\to 0
\end{equation}
where $\cM'=\cM\cap \prod_i V'_i$ and $\cM''$ is the image of $\cM\subset  \prod_i V_i$ in $\prod_i V''_i$.

One deduces from here  an exact sequence of $\bar\Q_p[G_{\Q_p}]$-modules:
\begin{equation}\label{splitsmod}
  \cM'/\gm\cM'\to  \cM/\gm\cM \to  \cM''/\gm\cM''\to 0
\end{equation}
where $G_{\Q_p}$ acts on $\cM''/\gm\cM''$ via $\psi''$ and on $\cM'/\gm\cM'$ via $\psi'$. 
Since  $\cM/\gm\cM\simeq \psi'\oplus \psi''$ as  $\bar\Q_p[G_{\Q_p}]$-modules and  since $\psi'\neq \psi''$
by assumption, it follows that $\dim_{\bar\Q_p}(\cM''/\gm\cM'')=1$, hence  by Nakayama's lemma,  $\cM''$ is free of rank one over 
$\cT$. It follows that (\ref{splits}) splits, hence  the first homomorphism in (\ref{splitsmod}) is injective and $\cM'$ is free of rank one over $\cT$ by the same argument as for $\cM''$. 
\end{proof}

Let $\Lambda$ be the completed local ring of $\cW$ at $\kappa(x)$ and let  $\gm_\Lambda$ be its maximal ideal.
It is isomorphic to a power series ring in one variable over $\bar\Q_p$. 
The weight map $\kappa$ induces a  finite and flat  homomorphism $\Lambda \to \cT$ of local reduced complete rings (cf.~ \cite[Proposition 7.2.2]{coleman-mazur}). 
The algebra of the fiber of $\kappa$ at $x$ is a local  Artinian  $\bar \Q_p$-algebra given by  
 \begin{equation}\label{fiber}
\cT' := \cT/m_\Lambda\cT.
 \end{equation}

\section{Proof of the Main Theorem and its Corollary}\label{main}

For  the rest of the paper we will assume that  $f$ is regular at $p$. 

Consider the deformation functors $\cD$ and $\cD'$ from \S\ref{deforms} associated to $\rho$ and the unramified character $\psi''$ of $G_{\Q_p}$ sending $\Frob_p$ to $\alpha$. The functor $\cD$ is pro-representable by a local Noetherian $\bar\Q_p$-algebra $\cR$, while $\cD'$  is representable by a local Artinian ring $\cR'$ with  residue field $\bar\Q_p$.  As usual,  one has  isomorphisms of $\bar \Q_p$-vector spaces 
$$t_{\cD}^\ast \cong \gm_{\cR}/\gm_{\cR}^2 \text{ and } t_{\cD'}^\ast\cong \gm_{\cR'}/\gm_{\cR'}^2,$$
where $W^\ast$ is the dual of a vector space $W$.

By Proposition~\ref{rhott}, the representation $\rho_{\cT}$ yields a continuous homomorphism of local complete Noetherian $\bar \Q_p$-algebras
\begin{equation}\label{surj}
\cR \to \cT.
\end{equation}

Let $A$ be any local Artinian ring with maximal ideal $\gm_A$ and residue field $A/\gm_A=\bar \Q_p$. A deformation of $\det(\rho)$ to $A^\times$ is equivalent to 
a continuous homomorphism from $G_{\Q}$ to $1+\gm_A$. By class field theory (and since $1+\gm_A$ does not contain elements of finite order), the latter are elements of:
$$\Hom((\Z/N)^\times \times \Z_p^\times, 1+\gm_A)= \Hom(1+q\Z_p, 1+\gm_A) \text{,  where} $$ 
 $$ q=p\text{ if }p>2\text{ , and } q=4\text{ if }p=2.$$
It follows that the universal deformation ring of $\det(\rho)$ is given by  $\bar\Q_p[[1+q\Z_p]]$ which 
 endows   $\cR$ with a  structure of $\bar \Q_p[[1+q\Z_p]] $-algebra.   
  By definition  $\det(\rho_{\cR})$
  factors through $\bar \Q_p[[1+q\Z_p]]^\times \to \cR^\times$ and $$\cR'=\cR/\gm_{\bar \Q_p[[1+q\Z_p]]}\cR.$$  
   
      Since  the Langlands correspondence  relates the determinant to the central character, the homomorphism (\ref{surj}) is  $\Lambda$-linear, in the sense that there
     exists an  isomorphism $\bar \Q_p[[1+q\Z_p]]\overset{\sim}{\longrightarrow} \Lambda$ fitting into 
     the following  commutative diagram: 
  $$\xymatrix{  \bar \Q_p[[1+q\Z_p]]    \ar[r] \ar@{->}[d] & \Lambda  \ar@{->}[d] \\
  \cR \ar@{->}[r] &\cT}.$$
 In particular, the homomorphism $\det(\rho_{\cT}) : G_{\Q} \to \cT^\times$ factors through $\Lambda^\times  \to  \cT^\times$ and $\cT'$ is the largest quotient of $\cT$ to which $\rho$ can be deformed with constant determinant.

\begin{prop} \label{propsurj}
 The homomorphisms (\ref{surj}) and (\ref{surjbis})  are surjective.  
\end{prop}

\begin{proof}
 Since $\cT$ is topologically generated  over $\Lambda$ by $U_p$ and  $T_\ell$ for  $\ell \nmid Np$  it suffices to check that those elements are in the image of $\cR$. 
For $\ell \nmid Np$,  $T_\ell=\Tr \rho_{\cT}(\Frob_\ell)$  is the image of the  trace of $\rho_{\cR}(\Frob_\ell)$. 

Recall that the restriction to $G_{\Q_p}$ of the universal deformation
 $\rho_{\cR} : G_{\Q} \to \GL_2(\cR)$   is an extension of an unramified 
character $\psi''_{\cR}$ lifting $\psi''$ by a character $\psi'_{\cR}$. By Proposition~\ref{rhott},  $U_p$ is the image of $\psi''_{\cR}(\Frob_p)$.\end{proof}

Reducing (\ref{surj}) modulo $\gm_\Lambda$ yields a natural surjective homomorphism of local Artinian $\bar \Q_p$-algebras: 
\begin{equation}\label{surjbis}
\cR' \to \cT'.
\end{equation}

\begin{thm} \label{thmiso} 
Both  (\ref{surj}) and (\ref{surjbis}) are isomorphisms and $\cT$ is regular.
\end{thm}
\begin{proof}
 It follows from Theorem~\ref{comptangent} that the tangent space $t_{\cD}$ of $\cR$ has  dimension $1$. Since $\cT$ is equidimensional of Krull dimension $1$, Proposition~\ref{propsurj} implies that (\ref{surj}) is an isomorphism of regular local rings of Krull dimension $1$. Since  (\ref{surjbis}) is the obtained by
 reducing (\ref{surj}) modulo $\gm_\Lambda$, it is an isomorphism too. 
  \end{proof}

\begin{proof}[Proof of Theorem~\ref{main-thm}]
We have already proved the first part: the eigencurve is smooth at the point $x$ corresponding to $f_\alpha$. 
Moreover, $\kappa$ is \'etale at $x$ if, and only if, the degree of $\kappa$ at $x$, that is 
the dimension of $\cT'$ over $\bar \Q_p$,  equals $1$. 
Since  (\ref{surjbis})  is an isomorphism, $\dim_{\bar \Q_p} \cR'= \dim_{\bar \Q_p} \cT'= 1$ if,  and only if,  $t_{\cD'}=0$.  To complete the proof, it suffices then to apply  Theorem~\ref{comptangent}.
 \end{proof}

\begin{proof}[Proof of Corollary~\ref{main-cor}]
Assume that  $f$ has CM by a field $K$ in which $p$ splits. We need to prove that any irreducible component of $\cC$ containing  $f_\alpha$ also has CM by $K$. Let  $\cC_K$ be the eigencurve constructed using Buzzard's eigenvariety machine \cite{buzzard} using  the submodule of overconvergent modular forms with CM by $K$. By the eigenvariety machine, $\cC_K$ is equidimensional of dimension one, and there is a natural closed immersion from $\cC_K$ to $\cC$ by a simpler analogue of the main result of \cite{chenevier-JL}. Since  $f_\alpha$ belongs to the image of $\cC_K$ and since $\cC$ has a unique irreducible component containing $f_\alpha$, it follows that this component is the 
 image of $\cC_K$.
  \end{proof}

 \section{The full eigencurve}\label{full}

Keep the notations from \S\ref{hecke}.
 Let $\cC^{\full}$ be the $p$-adic eigencurve of tame level $N$ constructed 
using the Hecke operators $T_\ell$ for  $\ell \nmid Np$ and $U_\ell$ for $\ell \mid Np$.
It  comes with a locally finite surjective morphism  $\cC^{\full} \to \cC$ compatible with all other structures (which is not an isomorphism when $N>1$), yielding by composition a two dimensional  pseudo-character $G_{\Q} \to \cO(\cC^{\full})$.
 There is a natural bijection between $\cC^{\full}(\bar \Q_p)$ and the set of systems of eigenvalues of overconvergent eigenforms with finite slope of tame level dividing $N$ and weight in $\cW(\bar \Q_p)$,  sending $y$ to the system of eigenvalues $T_\ell(y)$, $\ell \nmid Np$, and  $U_\ell(y)$, $\ell \mid Np$. 

  Let $\cT^\full$ be the completed local ring of $\cC^{\full}$ at 
the point defined by $f_\alpha$, and let  $\gm^\full$ be its maximal ideal.  Put  $\cT'^\full := \cT^\full/m_\Lambda\cT^\full$. The morphism $\cC^{\full} \to \cC$ yields a homomorphism $\cT \to \cT^\full$ and by Proposition \ref{rhott} 
there exists a continuous representation 
$$\rho_{\cT}^\full : G_{\Q} \to \GL_2(\cT^\full),$$  
such that $\Tr \rho_{\cT}^\full(\Frob_\ell) = T_\ell$,  for all $\ell \nmid Np$. Moreover $\rho_{\cT}^\full$ is ordinary at $p$ and its  
 reduction  modulo $\gm^\full$  is isomorphic to $\rho$.

Denote by  $\cM$ the free rank two $\cT^\full$-module on which  $\rho_{\cT}^\full$ acts.  

\begin{prop} \label{monodromy} Let $\ell$ be a prime dividing $N$. Then  $\rho_{\cT}^{\full}(I_\ell)$ is finite. Moreover:
\begin{itemize} \item[(i)] If $a_\ell \neq 0$, then $\cM^{I_\ell}$  is a free rank one  direct summand of $\cM$ on which 
 $\rho_{\cT}^\full(\Frob_\ell)$ acts as multiplication by $U_\ell$.
\item[(ii)] If $a_\ell = 0$, then $\cM^{I_\ell}=0$ and $U_\ell=0$ in $\cT^\full$.
 \end{itemize}
\end{prop}
\begin{proof} 
To prove that $\rho_{\cT}^\full(I_\ell)$ is finite, it is enough to prove that $\rho_{\cT}(I_\ell)$ is finite since $\rho_\cT^\full$ factors, by construction, through $\rho_\cT$.
Fix a non-zero continuous homomorphism $t_p : I_\ell \to \Z_p$ .
By Grothendieck's monodromy theorem for families (cf.~\cite[Lemma 7.8.14]{bellaiche-chenevier-book}), there exists a unique nilpotent matrix $\mathcal{N} \in M_2(\cT)$ 
such that  $\rho(g) = \exp(t_p(g) \mathcal{N})$ on an open subgroup of $I_\ell$. Thus it suffices to show that $\mathcal{N}=0$. Let $s$ be a generic point of $\Spec(\cT)$, $\rho_s$ the attached representation and $\mathcal{N}_s$ its monodromy operator. If $\mathcal{N}_s \neq 0$, then as is well known $\rho_{s|G_\ell}$ is an extension of $1$ by the $p$-adic cyclotomic character, hence the same is true for   $\rho_{\cT} \mod \gm \simeq \rho$, which is impossible since $\rho$ has finite image and $\ell$ is not a root of unity. Thus $\mathcal{N}_s=0$ for all generic points $s$. Since $\cT$ is reduced, this easily implies that $\mathcal{N}=0$.

Hence $\rho_{\cT}(I_\ell)$ is finite. It follows that the submodule $\cM^{I_\ell}$ 
 is a direct summand of $\cM$ (as it is the image of the projector $P=\frac{1}{|\rho^\full_{\cT}(I_\ell)|} \sum_{g \in \rho_{\cT}(I_\ell)} g$), 
and that the natural homomorphism  $\cM^{I_\ell} /\gm^\full  \cM^{I_\ell}  \to (\cM/\gm^\full \cM)^{I_\ell}$
is an isomorphism (it is injective since $\cM^{I_\ell}$ is direct summand and surjective, because  $z \in  (\cM/\gm^\full \cM)^{I_\ell}$ is the image of $P(z')$, where $z'$ is any lift of $z$ in $\cM$).
The space $(\cM/\gm^\full \cM)^{I_\ell} = \rho^{I_\ell}$ has dimension $1$ if $a_\ell \neq 0$, and $0$ if $a_\ell = 0$. 
By Nakayama, $\cM^{I_\ell}$ is $0$ if $a_\ell = 0$, and is free of rank one if $a_\ell \neq 0$. 

For the assertions concerning $U_\ell$, choose an affinoid neighborhood $U$ of  
the point defined by $f_\alpha$ in $\cC^\full$ such that there exists a
representation $\rho_U : G_{\Q} \to \GL(\cM_U)$, where $\cM_U$ is a free module over $\cO(U)$ of rank $2$, 
such that $\Tr \rho_U  : G_{\Q} \to \cO(U)$ is the natural pseudo-character (in particular  $\cM_U \otimes_{\cO(U)} \cT^\full \simeq \cM$ as $G_{\Q}$-modules). By standard arguments, there exists a Zariski-dense set of classical points $y \in U$ such that $(\cM_{U,y})^{I_\ell}= (\cM_U^{I_\ell})_y$, where the subscript $y$ means the fiber at $y$. If $\cM^{I_\ell}$ has rank one (resp. zero), so has 
$\cM_U^{I_\ell}$ up to shrinking $U$, and  so has  $(\cM_{U,y})^{I_\ell}$ for $y$ as above, 
 meaning that the action of $\Frob_\ell$ on $(\cM_{U,y})^{I_\ell}$
is the multiplication by $U_\ell(y)$ (resp. that $U_\ell(y)=0$). By Zariski density of those $y$'s, this means that $\rho_U(\Frob_\ell)$ acts by 
multiplication by $U_\ell$ on $\cM_U$ (resp. that $U_\ell = 0$ in $\cO(U)$), which implies the desired results.
\end{proof}

Let us also introduce a deformation problem   $\cD^\full$ where 
$$\cD^\full(A) = \{  \rho_A \in \cD(A),   \rho_A^{I_\ell} \text{ is a free of rank one over } A, \text{ for all } \ell \mid N, a_\ell \neq 0  \}.  $$ 
It is clear that $\cD^\full$ is a representable by a complete local ring $\cR^\full$ and that there is a natural surjective local ring homomorphism $\cR \twoheadrightarrow \cR^\full$.

By Proposition \ref{monodromy} the representation $\rho_{\cT}^\full$ defines a point of $\cD^\full$, hence a homorphism of local Artinian $\bar \Q_p$-algebras
\begin{equation}\label{RT-full}
\cR^\full \to \cT^\full.
\end{equation}

\begin{thm}  \label{local-isom}
The map (\ref{RT-full}) is an isomorphism. Moreover   $\cC^{\full} \to \cC$ is an isomorphism locally at 
the point defined by $f_\alpha$ and $\cC^{\full}$ is smooth at   that point.
 \end{thm}
\begin{proof}
We show first that (\ref{RT-full})  is surjective. 
By Proposition \ref{propsurj} it is enough to prove that $U_\ell$ belongs to the image for  $\ell\mid N$. If $a_\ell=0$, then by Proposition \ref{monodromy}(ii) $U_\ell=0$ in $\cT^\full$  and there is nothing to prove. If $a_\ell \neq 0$, then by Proposition \ref{monodromy}(i) $U_\ell$ is the image of the element of $\cR^\full$ by which $\Frob_\ell$ acts on the free rank one module   $\rho_{\cR^\full}^{I_\ell}$.

The composition $\cR \twoheadrightarrow \cR^\full \twoheadrightarrow \cT^\full$ is surjective as a composition of two surjective maps, and factors through $\cT$, 
yielding a surjective homomorphism $\cT \twoheadrightarrow \cT^\full$. 
Since $\cT \to \cT^\full$ is  injective by definition, it is an isomorphism. 
Equivalently, since $\cT^\full$ is equidimensional of Krull dimension one, and $\cT$ is regular by Theorem~\ref{thmiso}, one has $\cT \simeq \cT^\full$ and $\cR^\full \simeq \cT^\full$. 
 \end{proof}

\bibliographystyle{siam}

\end{document}